\documentclass[leqno,12pt]{article} 
\usepackage{amsmath, amssymb}
\usepackage{amsthm} 
\theoremstyle{plain} 
\newtheorem{theorem}{\indent\sc Theorem}[section]
\newtheorem{lemma}[theorem]{\indent\sc Lemma}

\theoremstyle{definition} 

\newtheorem{remark}[theorem]{\indent\sc Remark}

\makeatletter
\def\address#1#2{\begingroup
	\noindent\parbox[t]{7.8cm}{%
		\small{\scshape\ignorespaces#1}\par\vskip1ex
		\noindent\small{\itshape E-mail address}%
		\/: #2\par\vskip4ex}\hfill%
	\endgroup}%
\makeatother
\title{\uppercase{Rellich type inequalities on the half line in the framework of equalities}} 
\author{
	\bigskip \\
	\textsc{Yi C. Huang and Fuping Shi} 
}
\date{} 
\def\R{{\mathbf R}}
\def\H{{\mathbf H}}
\def\L{{\mathfrak L}}
\def\D{{\mathrm d}}
\def\C{{\mathbf C}}

\begin{document}
	\maketitle
	\footnote{ 
		2020 \textit{Mathematics Subject Classification}.
		Primary 26D10; Secondary 46E35, 35A23
	}
	\footnote{ 
		\textit{Key words and phrases}.
		Rellich inequalities, Hardy inequalities, radial derivative, polar coordinates.
	}
	\footnote{ 
		Research of the authors is supported by the National NSF Grant of China (no. 11801274).
	}
	
	\begin{abstract}
		We establish in this paper some weighted Hardy and Rellich type inequalities on the half line in the framework of equalities, extending recent results proved by Machihara-Ozawa-Wadade and Bez-Machihara-Ozawa. In particular, the one-dimensional classical Rellich inequality is proved in the framework of equality.
	\end{abstract}
	
	\section{Introduction and main results}
	
	Let $n\ge2$. It is well-known that the standard Laplacian on ${\mathbb{R}}^{n}$ decomposes as \[\Delta=\Delta_r+\frac{1}{|x|^2}\Delta_{{\mathbb{S}}^{n-1}},\] where $\Delta_r$ denotes the radial Laplacian and $\Delta_{{\mathbb{S}}^{n-1}}$ denotes the Laplacian-Beltrami operator on the sphere ${\mathbb{S}}^{n-1}$. More precisely, the radial Laplacian is given by
	\[\Delta_rf=\partial_r^2f+\frac{n-1}{|x|}\partial_rf,\]
	where $\partial_r$ denotes the radial derivative $\partial_r=\frac{x}{|x|}\cdot\nabla$. 
		
	Let $\|\cdot\|_2$ denote the $L^2$-norm on $\mathbb{R}^n$ with respect to Lebesgue measure. Let $\R _Q=\frac{Q(Q-4)}{4}$, where $Q\in\mathbb{R}$. The following ``Rellich type equalities'' were recently proved by Machihara-Ozawa-Wadade in \cite{ref2} and Bez-Machihara-Ozawa in \cite{ref1}. 	
	
	\begin{theorem}[Bez-Machihara-Ozawa]
		For $n\ge2$ and $f\in C_0^{\infty}(\mathbb{R}^n\backslash\{0\})$ we have	
			\begin{align}\label{O1}
				\R_n^2\left\|\frac{f}{|x|^2}\right\|_2^2&=\|\Delta_rf\|_2^2-\left\|\Delta_rf+\R_n\frac{f}{|x|^2}\right\|_2^2-2\R_n\left\|\frac{f^*}{|x|}\right\|_2^2\\\label{O1'}
				&=\|\Delta_rf\|_2^2-\left(1+\frac{1}{2}\R_n\right)\left\|\Delta_rf+\R_n\frac{f}{|x|^2}\right\|_2^2+\frac{1}{2}\R_n\left\|f^\#\right\|_2^2,
			\end{align}	
		where $f^*=\partial_rf+\frac{n-4}{2|x|}f$ and $f^{\#}=\partial_r^2f+(n-3)\frac{\partial_r f}{|x|}+\frac{(n-4)^2}{4}\frac{f}{|x|^2}$.	
	\end{theorem}
	
		Note that the above results can be in fact rephrased on the ray $\mathbb{R}_+:=(0,\infty)$ since they only involve the radial derivative. The original equalities \eqref{O1} and \eqref{O1'} then follow from the polar decomposition formula and the radially reduced equalities (with $L^2(r^{n-1}\D r)$-norm on $\mathbb{R}_+$). Moreover, in this radially reduced one-dimensional setting, we can consider weighted analogues. Let $\beta \in \mathbb{R}$ and denote by $\|\cdot\|_{L_\beta^2}$ the following weighted $L^2$-norm:
	\[\|f\|_{L_\beta^2}^2=\int_{0}^{\infty}|f(r)|^2r^\beta\mathrm d r.\]
	Associated to $\alpha \in \mathbb{R}$, we also consider the weighted second order derivative  \[\L_\alpha:=\partial_r^2+\frac{\alpha}{r}\partial_r.\]
Obviously $\L_{n-1}=\Delta_r$. 

For $\alpha, \beta \in\mathbb{R}$, let \[\C_{\alpha,\beta}=\frac{(\beta-3)(2\alpha-\beta+1)}{4}.\]
For $\L_{\alpha}$, we have the following analogues of results obtained in \cite{ref2} and \cite{ref1}.

	\begin{theorem}	
		For $f\in \mathbb{C}_0^{\infty}(\mathbb{R}_+)$ there hold
		\begin{equation}\label{R1}
			\C_{\alpha,\beta}^2\left\|\frac{f}{r^2}\right\|_{L_\beta^2}^2=\left\|\L_{\alpha}f\right\|_{L_\beta^2}^2-\left\|\L_{\alpha}f+\C_{\alpha,\beta}\frac{f}{r^2}\right\|_{L_\beta^2}^2
			-2\C_{\alpha,\beta}\left\|\frac{f^*}{r}\right\|_{L_\beta^2}^2
		\end{equation}
	and, if $\beta-\alpha-2\ne0$,
	\begin{equation}\label{R2}
		\begin{aligned}
				\C_{\alpha,\beta}^2\left\|\frac{f}{r^2}\right\|_{L_\beta^2}^2&=\left\|\L_{\alpha}f\right\|_{L_\beta^2}^2-\left(1+\frac{2\C_{\alpha,\beta}}{\left(\beta-\alpha-2\right)^2}\right)\left\|\L_{\alpha}f+\C_{\alpha,\beta}\frac{f}{r^2}\right\|_{L_\beta^2}^2\\
				&\qquad+\frac{2\C_{\alpha,\beta}}{\left(\beta-\alpha-2\right)^2}\left\|f^\#\right\|_{L_\beta^2}^2,
		\end{aligned}	
	\end{equation}
	where $f^*=f'+\frac{\beta-3}{2r}f$ and $f^\#=f''+\frac{\beta-2}{r}f'+\frac{(\beta-3)^2}{4r^2}f.$
	\end{theorem}
Notice that $\C_{n-1,n-1}=\R_n$, therefore Theorem 1.1 is a direct consequence of Theorem 1.2 when $\alpha=\beta=n-1$,  that is,
\begin{align}\label{O2}
	\R_n^2\left\|\frac{f}{r^2}\right\|_{L_{n-1}^2}^2&=\|\L_{n-1}f\|_{L_{n-1}^2}^2-\left\|\L_{n-1}f+\R_n\frac{f}{r^2}\right\|_{L_{n-1}^2}^2-2\R_n\left\|\frac{f^*}{r}\right\|_{L_{n-1}^2}^2\\  \label{O2'}
	&=\left\|\L_{n-1}f\right\|_{L_{n-1}^2}^2-\left(1+\frac{\R_n}{2}\right)\left\|\L_{n-1}f+\R_n\frac{f}{r^2}\right\|_{L_{n-1}^2}^2+\frac{\R_n}{2}\left\|f^\#\right\|_{L_{n-1}^2}^2
\end{align}
To recover \eqref{O1}-\eqref{O1'} from \eqref{O2}-\eqref{O2'}, it suffices to use polar coordinates
\[(r,\omega)=\left(|x|,\frac{x}{|x|}\right)\in\mathbb{R}_+\times\mathbb{S}^{n-1},\]
combined with integration over $\mathbb{S}^{n-1}$ and Fubini theorem.

The equalities \eqref{R1} and \eqref{R2} imply the following ``Rellich type inequalities''.
\begin{theorem}
	For all $\alpha,\beta \in \mathbb{R}$ and $f\in C_0^{\infty}(\mathbb{R}_+)$, there holds  \begin{equation}\label{R7}
		\C_{\alpha,\beta}^2\left\|\frac{f}{r^2}\right\|_{{L_\beta^2}}^2\le\|\L_\alpha f\|_{{L_\beta^2}}^2.
	\end{equation}

\end{theorem}
 For $n\ge2$, by taking $\alpha=\beta=n-1$ in Theorem 1.3 and using polar coordinates and Fubini theorem as indicated above, we get the higher dimensional results in \cite{ref1}.
 \begin{theorem}[Bez-Machihara-Ozawa]
 	Let $n\ge2$. For $f\in C_0^{\infty}(\mathbb{R}^n\backslash\{0\})$, there holds 
 	\[\R_n^2\left\|\frac{f}{|x|^2}\right\|_2^2\le\|\Delta_rf\|_2^2.\]
 	
 \end{theorem}

\begin{remark}
	There is a norm comparsion between $\Delta_r$ and $\Delta$: if $n\ge3$, then
	\[\|\Delta_rf\|_2\le \|\Delta f\|_2\]
	holds for $f\in C_0^{\infty}(\mathbb{R}^n\backslash\{0\})$. 
	This combined with Theorem 1.4 gives a proof of classical Rellich inequalities (see \cite{R54, R69}) on $\mathbb{R}^n$ for $n\ge3$ in the framework of equalities. The one-dimensional Rellich inequality is included in our Theorem 1.2 and Theorem 1.3.
\end{remark}

\section{Proof of Theorem 1.2}
	\begin{proof}[Proof of \eqref{R1}]
		If $\beta=3$, the equality trivially holds. So we assume $\beta\ne3$ in the following. Denote the Hermitian product with respect to $\left\|\cdot\right\|_{L_\beta^2}$ by $\left\langle \cdot,\cdot\right\rangle$. Using integration by parts twice with respect to $\left\|\frac{f}{r^2}\right\|_{L_\beta^2}^2$ we have
		\begin{equation}\label{R3}
				\begin{aligned}
				\int_{0}^{\infty}\frac{|f(r)|^2}{r^4}r^{\beta}\mathrm d r
				&=\int_{0}^{\infty}f(r)\overline{f(r)}r^{\beta-4}\mathrm d r\\
				&=-\frac{2}{\beta-3}\Re\int_{0}^{\infty}f(r)\overline{f'(r)}r^{\beta-3}\mathrm d r\\
				&=\frac{2}{(\beta-2)(\beta-3)}\Re\int_{0}^{\infty}\left(|f'(r)|^2+f(r)\overline{f''(r)}\right)r^{\beta-2}\mathrm d r\\
				&=\frac{2}{(\beta-2)(\beta-3)}\left(\left\|\frac{f'}{r}\right\|_{L_\beta^2}^2+\Re\left\langle\frac{f}{r^2},{f''}\right\rangle\right).
			\end{aligned}
		\end{equation}		
We used $\beta\ne2$ in the third step and we will give details for the case $\beta=2$ later. 

The first term in the last expression on RHS of \eqref{R3} is rewritten as \begin{equation}\label{R4}
	\left\|\frac{f'}{r}\right\|_{L_{\beta}^2}^2=\left\|\left(\frac{f}{r}\right)'+\frac{f}{r^2}\right\|_{L_{\beta}^2}^2
	=\left\|\left(\frac{f}{r}\right)'\right\|_{L_{\beta}^2}^2+2\Re\left\langle\left(\frac{f}{r}\right)',\frac{f}{r^2}\right\rangle+\left\|\frac{f}{r^2}\right\|_{L_{\beta}^2}^2.
\end{equation}	
We apply Lemma 2.2 below with $f$ replaced by $\frac{f}{r}$ to obtain\[\left\|\left(\frac{f}{r}\right)'\right\|_{L_\beta^2}^2=\H_{\beta+1}^2\left\|\frac{f}{r^2}\right\|_{L_\beta^2}^2+\left\|\frac{f'}{r}+\frac{\H_{\beta+1}-1}{r^2}f\right\|_{L_\beta^2}^2.\]
An integration by parts gives 
\[2\Re\left\langle\left(\frac{f}{r}\right)',\frac{f}{r^2}\right\rangle=\int_{0}^{\infty}\frac{\left(\frac{|f|^2}{r^2}\right)'}{r}r^{\beta}\mathrm d r=-(\beta-1)\left\|\frac{f}{r^2}\right\|_{L_\beta^2}^2.\]
Substituting above two equalities into \eqref{R4} leads to
\[\left\|\frac{f'}{r}\right\|_{L_\beta^2}^2=(\H_{\beta+1}-1)^2\left\|\frac{f}{r^2}\right\|_{L_\beta^2}^2+\left\|\left(\frac{f}{r}\right)'+\H_{\beta+1}\frac{f}{r^2}\right\|_{L_\beta^2}^2.\]
The second term in the last expression on RHS of \eqref{R3} is rewritten as
	\begin{align}\label{Q1}
		\Re\left\langle\frac{f}{r^2},{f''}\right\rangle=\Re\left\langle\frac{f}{r^2},{f''+\frac{\alpha}{r}f'}\right\rangle-\alpha\Re\left\langle\frac{f}{r^3},{f'}\right\rangle,
	\end{align}
	where the last term is given by
	\[\Re\left\langle\frac{f}{r^3},{f'}\right\rangle=\frac{1}{2}\Re\left\langle\frac{1}{r^3},{(f\bar{f})'}\right\rangle=-\frac{(\beta-3)}{2}\left\|\frac{f}{r^2}\right\|_{L_\beta^2}^2.\]
If $2\alpha-\beta+1=0$, then $\C_{\alpha,\beta}=0$ and \eqref{R1} holds trivially. If not, then $\C_{\alpha,\beta}\ne0$, hence we can gather the equalities \eqref{R4} and \eqref{Q1} and get
	\[\left\|\frac{f}{r^2}\right\|_{L_\beta^2}^2=-\frac{1}{\C_{\alpha,\beta}}\Re\left\langle\frac{f}{r^2},{f''+\frac{\alpha}{r}f'}\right\rangle-\frac{1}{\C_{\alpha,\beta}}\left\|\frac{f'}{r}+(\H_{\beta+1}-1)\frac{f}{r^2}\right\|_{L_\beta^2}^2.\]
	By Lemma \ref{l1} below, we get the desired equality.
		
		For $\beta=2$, one can verify that
		\[\Re\left\langle\frac{f}{r^2},{f''}\right\rangle=\Re\left\langle\frac{f}{r},\frac{f'}{r}\right\rangle=\left\|\frac{f'}{r}\right\|_{L_2^2}^2,\]
		and thus by similar arguments we complete the proof of \eqref{R1}.\end{proof}
		
		For the completeness of the arguments we record the mentioned lemmas below.
		\begin{lemma}[Machihara-Ozawa-Wadade]\label{l1}
			Let $\mathsf{H}$ be a vector space with Hermitian scalar product $\left\langle\cdot,\cdot\right\rangle$. Let $a\in\mathbb{R}, c\ne0$ and $u,v\in \mathsf{H}$. Then
			\[\|u\|^2=-c\Re\left\langle u,v\right\rangle+a\Leftrightarrow\frac{1}{c^2}\|u\|^2=\|v\|^2-\left\|v+\frac{1}{c}u\right\|^2+\frac{2a}{c^2}.\]
		\end{lemma}
	The proof of this lemma is straightforward, noting that since $\left\langle\cdot,\cdot\right\rangle$ is Hermitian,
	\[\left\langle f-g,f-g\right\rangle=\|f\|_{\mathsf{H}}^2+\|g\|_{\mathsf{H}}^2-2\Re\left\langle f,g\right\rangle=\|f\|_{\mathsf{H}}^2+\|g\|_{\mathsf{H}}^2-2\Re\left\langle g,f\right\rangle.\]		
	
			Let $\H _Q=\frac{Q-2}{2}$, where $Q\in\mathbb{R}$. 
			The following can be viewed as a reduction to radial functions of the Hardy inequalities proved in \cite{MOW19} in the framework of equalities.
			
			\begin{lemma}
				For $\beta\in\mathbb{R}$ and $f\in C_0^{\infty}(\mathbb{R}_+)$ there holds \begin{equation}\label{H1}
					\H_{\beta+1}^2\left\|\frac{f}{r}\right\|_{L_\beta^2}^2=\left\|f'\right\|_{L_\beta^2}^2-\left\|f'+\H_{\beta+1}\frac{f}{r}\right\|_{L_\beta^2}^2.
				\end{equation}
			\end{lemma}
	
	\begin{proof}
		Note that the lemma when $\beta=1$ holds trivially. For $\beta\ne1$, let $u=-\frac{2}{\beta-1}f'$ and $v=\frac{f}{r}$. Using integration by parts as before, we obtain the expression
		\[\left\|v\right\|_{L_\beta^2}^2=\Re\int_0^{\infty}v\bar{u}r^{\beta}\D r.\]
		Then
		\begin{align*}
			\left\|u\right\|_{L_\beta^2}^2-\left\|v\right\|_{L_\beta^2}^2&=\left\|u\right\|_{L_\beta^2}^2-\left\|v\right\|_{L_\beta^2}^2+2\Re\int_0^{\infty}\left(|v|^2-v\bar{u}\right)r^{\beta}\D r\\
			&=\int_{0}^{\infty}|v-u|^2r^\beta \D r,
		\end{align*}
	which gives the desired equality.
	\end{proof}
\begin{remark}
	For $t\in\mathbb{R}$, by taking $\beta=t+2$ in \eqref{H1}, we obtain the weighted one-dimensional Hardy inequality (see, for example, Cassano and Pizzichillo \cite{CP18})
	\[\int_0^{\infty}|f'(r)|^2r^{t+2}\D r\ge\left(\frac{t+1}{2}\right)^2\int_{0}^{\infty}|f(r)|^2r^t\D r.\]
	As a consequence, we have the weighted Rellich inequality
	\[\int_0^{\infty}|f''(r)|^2r^{t+4}\D r\ge\left(\frac{t+3}{2}\right)^2\left(\frac{t+1}{2}\right)^2\int_{0}^{\infty}|f(r)|^2r^t\D r.\]
\end{remark}

\begin{proof}[Proof of \eqref{R2}]
	For $a\in\mathbb{R}$, we start by observing that
	\begin{equation}\label{R5}
		\frac{(\beta-1-a)^2}{4}\left\|\frac{g}{r^{1+\frac{a}{2}}}\right\|_{L_\beta^2}^2=\left\|\frac{g'}{r^{\frac{a}{2}}}\right\|_{L_\beta^2}^2-\left\|\frac{\beta-1-a}{2}\frac{g}{r^{1+\frac{a}{2}}}+\frac{g'}{r^{\frac{a}{2}}}\right\|_{L_\beta^2}^2
	\end{equation} 
holds for all $g\in C_0^{\infty}(\mathbb{R}_+)$. Indeed, this follows by using integration by parts to see that
\[2\Re\left\langle\frac{g}{r^{1+\frac{a}{2}}},\frac{g'}{r^{\frac{a}{2}}}\right\rangle=-(\beta-1-a)\left\|\frac{g}{r^{1+\frac{a}{2}}}\right\|_{L_\beta^2}^2.\]
If $\beta-1-a=0$, the equality trivially holds. By using Lemma \ref{l1}, we obtain \eqref{R5}. 

Now taking $a=2\alpha-\beta+3$ and $g$ given by
\[g(r)=\left(\frac{f'}{r}+(\H_{\beta+1}-1)\frac{f}{r^2}\right)r^{\frac{2\alpha-\beta+5}{2}},\]
we obtain that for $\beta-\alpha-2\ne0$,
\begin{equation}\label{R6}
		\left\|\frac{f'}{r}+\frac{\beta-3}{2r^2}f\right\|_{L_\beta^2}^2=\frac{1}{(\beta-\alpha-2)^2}\left\|\L_{\alpha}f+\C_{\alpha,\beta}\frac{f}{r^2}\right\|_{L_\beta^2}^2-\frac{1}{(\beta-\alpha-2)^2}	\left\|f^\#\right\|_{L_\beta^2}^2,	
\end{equation}
where $f^\#=f''+\frac{\beta-2}{r}f'+\frac{(\beta-3)^2}{4r^2}f.$ We see that \eqref{R2} follows from \eqref{R1} and \eqref{R6}.
\end{proof}
	
	\section{Proof of Theorem 1.3}

		If $\C_{\alpha,\beta}\ge0$, 
		then \eqref{R7} follows from \eqref{R1}. For $\beta-\alpha-2=0$, we have
		\[\C_{\alpha,\beta}=\frac{(\beta-3)^2}{4}.\]
	Therefore, if $\C_{\alpha,\beta}<0$, then $\beta-\alpha-2\ne0$. We can use the expression \eqref{R2} for 	$\C_{\alpha,\beta}^2\left\|\frac{f}{r^2}\right\|_{L_\beta^2}^2$ in Theorem 1.2. A direct calculation gives
	\begin{align*}
		1+\frac{2\C_{\alpha,\beta}}{\left(\beta-\alpha-2\right)^2}&=\frac{1}{2}+\frac{2\alpha\beta-\beta^2+4\beta-6\alpha-3+(\beta-(\alpha+2))^2}{2\left(\beta-\alpha-2\right)^2}
		\\&=\frac{1}{2}+\frac{(\alpha-1)^2}{2\left(\beta-\alpha-2\right)^2}>0,
	\end{align*} 
		thus \eqref{R7} follows. This completes the proof.

\bigskip
\address{ 
	School of Mathematical Sciences \\
	Nanjing Normal University \\
	Nanjing 210023 \\
	People's Republic of China
}
{Yi.Huang@njnu.edu.cn}
\address{
	School of Mathematical Sciences \\
	Nanjing Normal University \\
	Nanjing 210023 \\
	People's Republic of China
}
{ryan1165761719@foxmail.com}
\end{document}